\documentclass[letterpaper,11pt,reqno]{amsart}
\usepackage{amssymb}
\usepackage{amscd}
\usepackage{amsfonts}
\usepackage{amsmath}
\usepackage{amsthm}
\usepackage[all]{xy}
\usepackage{mathrsfs}
\usepackage{tikz}
\usepackage{etex}
\usepackage[paper = letterpaper, left = 30mm, right = 30mm, headsep = 7mm,
footskip = 10mm, top = 30mm, bottom = 30mm, footnotesep=5mm, headheight =
2cm]{geometry}

\usepackage{titletoc}
\titlecontents{part}[0pt]{}{\bfseries\chaptername\ \thecontentslabel\quad}{\bfseries}{\bfseries\hfill}
\usepackage{xr}
\usepackage[colorlinks=true]{hyperref}
\usepackage{url}
\usepackage{cite}
\usepackage{xcolor}
\renewcommand\thepart{\Roman{part}}

\def\Amp{\mathrm{Amp}}

\def\rank{\mathrm{rank}}

\def\Todd{\mathrm{Todd}}

\def\part#1{%
  \vskip .02\vsize 
  \refstepcounter{part}
  \addcontentsline{toc}{part}{Part \thepart:\ #1}
  {\centering\large \textbf{Part \thepart}. #1\par}%
  \vskip .01\vsize
}

\newcommand{\Hom}{\operatorname{Hom}}
\newcommand{\Pos}{\operatorname{Pos}}
\newcommand{\Hilb}{\operatorname{Hilb}}

\theoremstyle{plain}
\newtheorem{theorem}{Theorem}[section]
\newtheorem*{theorem*}{Theorem}
\newtheorem{lemma}[theorem]{Lemma}

\newtheorem{proposition}[theorem]{Proposition}
\newtheorem*{definition*}{Definition}

\numberwithin{equation}{section} 

\makeatletter
\def\ifdraft{\ifdim\overfullrule>\z@
  \expandafter\@firstoftwo\else\expandafter\@secondoftwo\fi}
\makeatother

\newsavebox{\ieeealgbox}

\newcommand{\N}{\mathbb{N}}

\newcommand{\Q}{\mathbb{Q}}
\newcommand{\R}{\mathbb{R}}

\def\hq{\hspace{-0.5mm}/\hspace{-0.14cm}/ \hspace{-0.5mm}}

\theoremstyle{definition}

\newtheorem{remark}[theorem]{Remark}
\newtheorem{comment2}[theorem]{Comment}
\newtheorem{example}[theorem]{Example}

\title[Moduli on higher-dimensional base manifolds]{Moduli of vector bundles on higher-dimensional\\base manifolds\\ --- \\Construction and variation}
\date{\today}
\author[D.~Greb]{Daniel Greb}
  \address{Essener Seminar f\"ur Algebraische Geometrie und Arithmetik, Fakult\"at f\"ur Mathematik, Universit\"at Duisburg-Essen, 45117 Essen, Germany}
  \email{daniel.greb@uni-due.de}
\author[J.~Ross]{Julius Ross}
  \address{Department of Pure Mathematics and Mathematical Statistics, University of Cambridge, Wilberforce Road, Cambridge, CB3 0WB, UK}
  \email{j.ross@dpmms.cam.ac.uk}
\author[M.~Toma]{Matei Toma}
\address{Institut de Math\'ematiques \'Elie Cartan, Universit\'e de Lorraine, B.P. 70239, 54506 Vandoeuvre-l\`es-Nancy Cedex,
France}
\email{matei.toma@univ-lorraine.fr}
\thanks{During the preparation of \cite{GrebToma}, DG was partially supported by an ``Eliteprogramm f\"ur Postdoktorandinnen und Postdoktoranden''-scholarship of the Baden-W\"urttemberg-Stiftung. JR is supported by an EPSRC Career Acceleration Fellowship (EP/J002062/1).}
\begin{document}

\begin{abstract} We survey recent progress in the study of moduli of vector bundles on higher-dimensional base manifolds. In particular, we discuss an algebro-geometric construction of an analogue for the Donaldson-Uhlenbeck compactification and explain how to use moduli spaces of quiver representations to show that Gieseker-Maruyama moduli spaces with respect to two different chosen polarisations are related via Thaddeus-flips through other ``multi-Gieseker''-moduli spaces of sheaves. Moreover, as a new result, we show the existence of a natural morphism from a multi-Gieseker moduli space to the corresponding Donaldson-Uhlenbeck moduli space.
\end{abstract}

\maketitle

\newcommand{\Space}{\,\,\,\,\,\,\,\,\,\,\,\,\,\,\,\,}

\section{Introduction}
Moduli spaces of sheaves play a central role in Algebraic Geometry: they provide intensively studied examples of higher-dimensional varieties (e.g.,~of hyperk\"ahler manifolds), they are naturally associated with the underlying variety and can therefore be used to define fine invariants of its differentiable structure, and they have found application in numerous problems of mathematical physics.

To obtain moduli spaces in the category of schemes, it is necessary to choose a semistability condition that selects the objects for which a moduli spaces is to be constructed. A vector bundle $E$ on a (smooth) complex curve is called \emph{semistable} if for any coherent subsheaf $F \subset E$ of intermediate rank the \emph{slope} of $F$, i.e., the degree of $F$ divided by the rank of $F$, is less than or equal to the slope of $E$. Note that the Hilbert polynomial of $E$ is completely controlled by its slope once the rank is fixed. With this definition in place, a projective moduli space for semistable vector bundles of fixed topological type exists and it parametrises so-called \emph{$S$-equivalence classes} of semistable vector bundles, see \cite{Ses67}. 

Starting in dimension two, the notion of Hilbert polynomial or degree of a given vector bundle is not well-defined, but depends on a further parameter, namely the class of a line bundle $H$ in the ample cone $\mathrm{Amp}(X)$ of $X$. Fixing such an $H$, one obtains two notions of semistability: a torsion-free coherent sheaf $E$ on $X$ is called \emph{slope-semistable} (with respect to $H$) if the slope of any coherent subsheaf of intermediate rank is less than or equal to the slope of $E$, and it is called \emph{Gieseker-Maruyama-} or just \emph{Gieseker-semistable} (with respect to $H$) if the reduced Hilbert polynomial (i.e., the Hilbert polynomial divided by the rank) of any non-trivial proper coherent subsheaf is less than or equal to the reduced Hilbert polynomial of $E$. These two notions are related by the fact that the slope is essentially the first non-trivial coefficient of the reduced Hilbert polynomial, when expanded in powers of $H$. While on the one hand slope-(semi)stability turns out to be very well behaved with respect to restrictions and tensor powers, and to have a very precise differential-geometric meaning (in terms of the existence of Hermite-Einstein metrics), on the other hand it was shown by Gieseker, Maruyama, and Simpson that once the topological invariants are fixed, there exists a projective moduli space $\mathcal{M}_H$ parametrising $S$-equivalence classes of Gieseker-semistable sheaves, see \cite[Chap.~4]{Bible}. Based on these results, two important questions arise naturally:
\vspace{0.2cm}
\begin{enumerate}
 \item[\textbf{(Q1)}] Can a meaningful moduli space $\mathcal{M}^\mu_H$ for slope-semistable sheaves be constructed as a projective scheme, and if this is the case, what is its relation to the Gieseker moduli space $\mathcal{M}_H$? 
 \item[\textbf{(Q2)}] How does the moduli space $\mathcal{M}_H$ vary when $H$ varies over the ample cone of $X$?
\end{enumerate}
\vspace{0.2cm}

It is the aim of this note to summarise known results in the surface case, to explain additional problems that arise in higher dimensions, and to discuss recent progress in both directions made by the authors in \cite{GRTI}, \cite{GRTII}, and \cite{GrebToma}.

\subsection{The surface case}
In dimension two, the problems outlined above have attracted a lot of interest due to the connection with Donaldson invariants \cite{DonaldsonPolynomialInvariants}, and questions (Q1) and (Q2) have been investigated by a number of authors. From this, a rather complete geometric picture has emerged:

(Q1) Given the class of an ample divisor $H$ on a smooth projective surface $X$, a moduli space for $H$-slope-semistable sheaves has been constructed by Jun Li \cite{JunLiDonaldsonUhlenbeck} and Le Potier \cite{LePotierDonaldsonUhlenbeck} using a certain determinant line bundle on the Gieseker moduli space $\mathcal{M}_H$. Two slope-semistable sheaves give the same point in $\mathcal{M}_H^\mu$ if and only if the double dual of the graded objects $Gr$ associated with their respective Jordan-H\"older filtration are isomorphic and if additionally the $0$-cycles obtained by taking the quotient of $Gr^{**}$ by $Gr$ coincide. Using this description, Jun Li and also Morgan \cite{Mor93} were able to show that the projective scheme $\mathcal{M}_H^\mu$ admits a homeomorphic map to the gauge-theoretic Donaldson-Uhlenbeck compactification $\mathcal{M}_H^{DU}$ \cite{DK90} of the moduli space of $H$-slope-stable vector bundles, thus endowing $\mathcal{M}_H^{DU}$ with an algebraic structure. Consequently, $\mathcal{M}_H^\mu$ provides a natural link between algebraic and differential geometry.

(Q2) The ample cone of $X$ supports a locally finite chamber structure given by linear rational walls so that the notion of slope-/Gieseker-semistability (and hence the moduli space) does not change within the chambers, see \cite{Qin}. Moreover, at least when the discriminant of the sheaves under consideration is sufficiently big, moduli spaces $\mathcal{M}_{H_1}$ and $\mathcal{M}_{H2}$ corresponding to two chambers separated by a common wall are birational, see e.g.~\cite[Prop. 4.C.6]{Bible}, and the passage from one to the other can be described explicitly in the following two ways:
First, Matsuki and Wentworth \cite{MatsukiWentworth} showed that $\mathcal{M}_{H_1}$ and $\mathcal{M}_{H2}$ are connected by a finite sequence of Thaddeus-flips passing through Gieseker moduli spaces $\mathcal{M}_{H + A_k}$ of (twisted-)semistable sheaves, where $H$ is a rational polarisation lying on the wall and $A_k$ is a suitable rational ample polarisation; similar results were also obtained for example by Ellingsrud and G\"ottsche \cite{EllingsrudGoettsche}. Second, it was shown (again in the case of large discriminant) by Hu and Li \cite{HuLi} that the situation is even simpler if one first passes from the Gieseker moduli spaces $\mathcal{M}_{H_1}$ and $\mathcal{M}_{H_2}$ to the corresponding moduli spaces $\mathcal{M}_{H_1}^\mu$ and $\mathcal{M}_{H_2}^\mu$ of slope-semistable sheaves: the latter spaces admit natural birational maps to the moduli space $\mathcal{M}_{H}^\mu$ of slope-semistable sheaves corresponding to a rational polarisation $H$ lying on the common wall between the chambers containing $H_{1}$ and $H_{2}$. We emphasise that both results depend in a crucial way on the fact that the polarisation $H$ lying on the wall is rational.

\subsection{New problems appearing in higher dimensions}
On the differential-geometric side of (Q1), generalising the Donaldson-Uhlenbeck compactification to higher-dimensional base manifolds, Tian \cite{Tian00} has constructed a natural gauge-theoretic topological compactification of the moduli space of stable holomorphic vector bundles, which is known to exist as an open separated complex-analytic subspace of the space of moduli space of simple bundles \cite{KosarewOkonek} or as a gauge-theoretic moduli space~\cite{KobayashiVectorBundles}: he completes the space of Hermitian-Yang-Mills connections by allowing so-called generalised self-dual instantons at the boundary. However, while it is generally expected that for any compact K\"ahler manifold and any moduli space of stable bundles there exists a complex-analytic compactification (even one that carries a natural K\"ahler structure, see for example the discussion in \cite[Sect.~3.2]{Tel08}), it is not known whether or not the topological compactification of Tian carries a complex structure. Jun Li's algebro-geometric construction in the surface case uses the Mehta-Ramanathan Theorem, which allows one to restrict slope-semistable sheaves to hyperplane sections, and the fact that on surfaces hyperplane sections are curves, on which the notion of slope-semistability has an interpretation in terms of Geometric Invariant Theory (GIT). A priori, this strategy runs into problems in higher dimensions, as there is no analogous GIT-interpretation of slope-semistability on a hyperplane section of dimension two or higher.

In the past, there have been several attempts to extend to higher-dimensional base manifolds the results concerning (Q2) describing the chamber structure and the variation of moduli spaces for sheaves on surfaces. However, in each of these approaches fundamental problems appear: 

On the one hand, Qin immediately remarks that his definition of ``wall'' will run into problems in higher dimensions, as the walls so defined will not be locally finite inside the ample cone, see \cite[Ex.~I.2.3]{Qin}. On the other hand, Schmitt \cite{Schmitt} looks at segments between two integer ample classes inside the real span of the ample cone, and defines a finite set of walls in the segment, so that on the complement of these walls the notion of semistability does not change. However, he also gives an example \cite[Ex.~1.1.5]{Schmitt} (which upon closer inspection turns out to be generic) of a Calabi-Yau threefold with Picard number two for which at least one such wall does not contain any rational point at all! More precisely, he exhibits threefolds $X$ with Picard number equal to two carrying rank two vector bundles $E$ that are slope-stable with respect to some integral ample divisor $H_0$ and unstable with respect to some other integral ample divisor $H_1$ such that the class $H_{\lambda}:=  (1-\lambda)H_0+\lambda H_1$ for which $E$ becomes strictly semistable is irrational. This irrationality can be traced back to the fact that the slope of a fixed sheaf $E$ changes in a non-linear way as the polarisation varies over the ample cone, cf.~Sections~\ref{subsect:slopesemistability} and \ref{subsect:divisors_to_curves}. Hence, comparing to the surface case, it is a priori unclear what should replace the moduli spaces corresponding to polarisations lying on walls that were used in Hu-Li's and Matsuki-Wentworth's work on question (Q2). 

At this point, we note that classes corresponding to points on non-rational walls inside the ample cone are K\"ahler, i.e., they can be represented by K\"ahler forms $\omega$ on the underlying projective manifold $X$. To define stability with respect to such an $\omega$, for a torsion-free sheaf $E$ consider the quantity
$$p_E(m) = \frac{1}{\rank(E)}\int_X ch(E) e^{m\omega} \Todd(X),$$
where $\Todd(X)$ is the Todd class of $X$ and $ch(E)$ is the Chern character of $E$.  We say that $E$ is \emph{(Gieseker-)(semi)stable with respect to $\omega$} if for all proper coherent subsheaves $E'\subset E$ we have $p_{E'}(m) < (\le) p_E(m)$ for all $m$ sufficiently large. 

In case $\omega$ represents the first Chern class of an ample line bundle $L$, the Riemann-Roch theorem states that $p_{E}(m)$ equals $\frac{1}{\rank(E)} \chi(E\otimes L^m)$, and so this generalises the notion of Gieseker-stability from integral classes to real classes, cf.~\cite[Sect.~3.2]{Tel08}. Hence, a third question of independent interest arises while investigating (Q1) and (Q2), cf.~\cite[Rem.~1.1.6]{Schmitt}:
\vspace{0.2cm}

\begin{enumerate}
\item[\textbf{(Q3)}] Does there exist a moduli space parametrising $\omega$-Gieseker-semistable sheaves, where $\omega$ is a K\"ahler form whose class lies in $\mathrm{Amp}(X)_\mathbb{R}$, the real span of the ample cone, or even outside $\mathrm{Amp}(X)_\mathbb{R}$?
Do moduli spaces of semistable sheaves on non-algebraic compact K\"ahler manifolds exist as complex spaces?
\end{enumerate}

\subsection{Connection with recent work on counting invariants and wall-crossing in triangulated categories}
Wall-crossing phenomena for moduli spaces have been investigated in a number of special geometric situations, most prominently in the study of counting invariants for Calabi-Yau threefolds such as Donaldson-Thomas invariants, see for example \cite{DonaldsonThomas} or \cite{RichardHolomorphicCasson}. In this context, the notions of slope- and Gieseker-semistability have been extended in many different directions, and Rudakov \cite{Rudakov} was one of the first to place these in the context of general abelian categories. This has since been built on for example by Joyce's epic [\citen{JoyceI}-\citen{JoyceIV}] and other works that use this (and generalisations to triangulated categories) to understand the wall crossing formulae that govern the change of generating functions for these invariants as the stability condition varies, see also \cite{JoyceSong}.

Most of the generalised stability conditions defined in this paper fall under Joyce's definitions; however, his work does not consider (or really has use for) the geometry of the coarse moduli spaces, and instead works throughout with the relevant kinds of stacks. 

On the other hand, also in this more abstract setup it is highly instructive to investigate the change in geometry of coarse moduli spaces in more detail. For example, in \cite{StoppaThomas} the fact that two moduli spaces are related by a variation of GIT was used to prove a relationship between certain generating functions for MNOP and stable pair invariants, and in \cite{KiemLi} a master space construction is used to study the change in the degree of virtual cycles on moduli stacks of stable objects in the derived category of a projective manifold. 

The results presented in Section~\ref{subsect:MultiVariation} in particular apply to moduli of sheaves on Calabi-Yau threefolds, and it would be interesting to investigate whether they have interesting consequences in the theory of (generalised) Donaldson-Thomas invariants.

That the classical and the modern approaches interact nicely is also shown by recent work of Bertram, who reinterprets and clarifies the work of Matsuki-Wentworth in the language of Bridgeland stability conditions on the bounded derived category of the given surface, see \cite{Bertramb} and \cite{Bertram}. It is an interesting open question whether an analogous interpretation can be found for the work presented in Section~\ref{subsect:quivervariation} below.

\subsection*{Acknowledgements} Daniel Greb wants to thank the organisers of VBAC 2014, especially Alexander Schmitt, for the invitation to participate and give a talk in the conference, which took place at Freie Universit\"at Berlin in September 2014. 

The authors are grateful to Adrian Langer for pointing out and correcting an error in the original formulation of Theorem 3.2. Moreover, they want to thank Valery Alexeev for his insightful comments that are summarised in Remark~\ref{rem:Alexeev} and the referee for constructive remarks. 

\subsection*{Conventions}
For simplicity, unless mentioned otherwise, we always assume the base space $X$ of the sheaves considered to be a smooth projective variety over the complex numbers $\mathbb{C}$, although many parts of the discussion can be carried out on more general schemes. 

\section{Slope-semistability and boundedness}

\subsection{Slope-semistability}\label{subsect:slopesemistability}
While classically slope-semistability was considered as being given by the choice of an ample divisor class, our point of view is different, and we consider curve classes instead:

Let $\alpha$ be a class in $N_1(X)_{\mathbb R} = N_1(X)\otimes \mathbb{R}$, the space of $1$-cycles on $X$ modulo numerical equivalence. A coherent sheaf $E$ on $X$ is called \emph{(semi)stable with respect to $\alpha$} or simply  \emph{$\alpha$-(semi)stable} if it is torsion-free, and if additionally for any proper non-trivial coherent subsheaf $F$ of $E$ we have 
 \begin{equation}
\mu_{\alpha}(F) := \frac{c_1(F) \cdot \alpha}{\mathrm{rank}(F)} < (\leq) \frac{c_1(E) \cdot \alpha}{\mathrm{rank}(E)} = \mu_{\alpha}(E).\label{def:slopecurveclass}
\end{equation}
The quantity $\mu_{\alpha}(F)$ is called the \emph{slope} of $F$ with respect to $\alpha$. For the notion of semistability to have reasonable properties, some positivity of $\alpha$ must be assumed (for instance it should at least be movable in the sense of \cite{BDPP}).  For us, the following special cases will be important:
\begin{enumerate}
\item $\alpha = [H]^{n-1}$, where $[H]\in N^{1}(X)$ is an integer ample class and $n=\dim X$.  Then, $\mu_{\alpha}$ is the slope in the classical sense, and we refer to this as \emph{$H$-(semi)stability}.
\item More generally, we could take $\alpha = [\omega]^{n-1}$ where $[\omega]$ is a K\"ahler class in $H^{1,1}(X, \mathbb{R})$, which we shall refer to as \emph{$\omega$-(semi)stability}.
\item We are also interested in the case $\alpha = [H_{1}].[H_{2}].\cdots.[H_{n-1}]$, where the $[H_{i}]$ are (integral) ample classes in $N^{1}(X)$, the space of divisors moduli numerical equivalence, and will refer to this as (semi)stability with respect to the \emph{multipolarisation} $(H_{1},\ldots,H_{n-1})$.  
\end{enumerate}
In all of these cases, a semistable sheaf will be called \emph{polystable} if it is a direct sum of stable sheaves.

The study of semistability with respect to a multipolarisation was introduced in \cite{Miyaoka}, and has later been extended to include a discussion of semistability with respect to arbitrary movable curve classes \cite{CaPe11}. Semistability (in this sense) of the tangent sheaf of a projective manifold or a projective variety with singularities as they appear in the Minimal Model Program has played an important role in higher-dimensional classification theory; in addition to the papers already mentioned see for example \cite{KebekusSolaCondeToma} or \cite{bbdecomp}. For more information we refer the reader to \cite{GKP14}, where some foundational results about semistable sheaves are established in this general context. 

\subsection{From divisors to curves}\label{subsect:divisors_to_curves}
Many previous efforts to study the variation of slope-stability consider the case (1) where $\alpha = [H]^{n-1}$ and  consider the change of semistabilty as $[H]$ varies over the real span $\mathrm{Amp}(X)_\mathbb{R}$ of the ample cone inside $N^1(X)_{\mathbb{R}} = N^1(X)\otimes \mathbb{R}$. Looking at the defining equation \eqref{def:slopecurveclass}, it becomes clear that in case $\dim X \geq 3$ the slope $\mu_{[H]}(E)$ of a fixed sheaf $E$ changes in a non-linear way with $[H]\in \mathrm{Amp}(X)_\mathbb{R}$, leading to the pathologies of Schmitt previously discussed.  In the following paragraphs, we explain how this problem can be resolved if one considers curve classes rather than divisor classes. As a first indication in this direction, one immediately notices that the slope $\mu_\alpha(E)$ of a given sheaf $E$ varies linearly with $\alpha \in N_1(X)_\mathbb{R}$, i.e., the non-linearity issue encountered before disappears.

The first crucial step is to realise that by looking at curve classes one does not lose information concerning divisor classes. For this, we consider the natural map
\begin{equation*}
 p_{n-1}: \mathrm{Amp}(X)_\mathbb{R} \to N_1(X)_{\mathbb R} \quad     \alpha \mapsto \alpha^{n-1}
\end{equation*}
and let $\mathrm{Pos}(X)_\mathbb{R}$ be the image of $p_{n-1}$ in $N_1(X)_{\mathbb R}$. We note that $\mathrm{Pos}(X)_{\mathbb R}$ is a generally non-convex, real cone contained in the interior of the movable cone. We call it the \emph{positive cone} of $X$.

\begin{proposition}[Injectivity of power maps]\label{P}
The set $\mathrm{Pos}(X)_{\mathbb R}$ is open in $N_1(X)_{\mathbb R}$, and the map $p_{n-1}\colon~\alpha\mapsto \alpha^{n-1}$ is a homeomorphism from $\mathrm{Amp}(X)_\mathbb{R}$ to $\mathrm{Pos}(X)_\mathbb{R}$.
\end{proposition}
The rather elementary proof given in \cite[Sect.~6.3.1]{GrebToma} is based on the Khovanskii-Teissier inequalities and the Hodge Index Theorem. 

\subsection{Boundedness}

The second crucial step in our approach is to obtain boundedness results for families of sheaves that are semistable with respect to curve classes. The proof of the classical case $\alpha = [H]^{n-1}$ of this result can be found for example in \cite[Thm.~3.3.7]{Bible}, and uses in an essential way that $H$ is integral so one is able to take hyperplane sections. When considering curve classes instead of divisors, new ideas are needed, as one encounters the following problems:
\begin{enumerate}
 \item To study variation of slope stability one needs boundedness of the set of sheaves (with given topological type) that are semistable with respect to a curve class that is allowed to vary in some given compact subset $K\subset N_{1}(X)_{\mathbb R}$.  Even if we restrict to $K$ being a polyhedral subset of $\Pos(X)_{\mathbb R}$, such a set will contain points that are not of the form $[H]^{n-1}$ for any ample ($\mathbb Q$)-class $H$.

\item As the slope is linear in $N_{1}(X)_{\mathbb R}$, it is natural to consider \emph{convex} subsets $K\subset N_{1}(X)_{\mathbb R}$.  So, even if originally interested in $\Pos(X)_{\mathbb R}$, we are naturally led to study its convex hull, which will contain points that are far away from being of the form $[H]^{n-1}$ for some ample class $[H]\in \mathrm{Amp}(X)_\mathbb{R}$.
\end{enumerate}

The first issue is dealt with by the following result \cite[Thm.~6.8]{GRTI}:
\begin{theorem}[Boundedness I]\label{thm:K^+}
 Let $X$ be a smooth $n$-dimensional projective variety over an algebraically closed field $k$ (of arbitrary characteristic). Let $K \subset \mathrm{Pos}(X)_\mathbb{R}$ be a compact subset. Then, the set of torsion-free sheaves $E$ with fixed Chern classes that are slope-semistable with respect to some class in $K$ is bounded. 
\end{theorem}

The idea of the proof is to connect a class $\alpha \in \mathrm{Pos}(X)_\mathbb{R}$ to a class of the form $[H]^{n-1}$, where $H$ is a integral ample divisor, and by analysing the change of semistabilty along the resulting segment using the following lemma (``proving boundedness by wall-crossing''):
\begin{lemma}
Let $\alpha$ be any class in $N_1(X)_{\mathbb R}$ and $H \in\Amp(X)_\Q$. If a torsion-free sheaf $E$ is slope-semistable with respect to $\alpha$ but not
 with respect to $H$, then $E$ is properly semistable with respect to a class $\alpha_t:=(1-t)\alpha+t[H]^{n-1}$, for some $t\in{[0,1)}$.
\end{lemma}
From this, the proof of Theorem \ref{thm:K^+} proceeds by bounding the slope of destabilising subsheaves with the help of the Hodge Index Theorem and Langer's version of the Bogomolov's inequality, reducing the problem to the classical boundedness result for $[H]^{n-1}$.    In fact, the same result holds when in case the ground field is $\mathbb{C}$ the positive cone is replaced by the K\"ahler cone of $X$, see \cite[Prop.~6.3]{GrebToma}.
\begin{remark}
 As the proof uses the Hodge Index Theorem, this is one of the places in this work where the smoothness assumption is essential.
\end{remark}

The next theorem \cite[Cor.~6.12]{GRTI} deals with the second issue, at least in a number of special cases.
\begin{theorem}[Boundedness II]\label{thm:boundedness_II}
  Let $X$ be a smooth projective variety over an algebraically closed field of arbitrary characteristic and let $K$ be the convex hull of $[H_{1}]^{n-1},\ldots,[H_{N}]^{n-1}$, where each $H_{i}$ is an ample integral divisor on $X$.  Suppose in addition that
  \begin{enumerate}
  \item the rank of the torsion-free sheaves under consideration is at most two, or
  \item the dimension of $X$ is at most three, or
  \item the Picard rank of $X$ is at most two.
  \end{enumerate}
Then, the set of torsion-free sheaves $E$ with fixed Chern classes that are slope-semistable with respect to some class in $K$ is bounded.
\end{theorem}
The main difficulty in proving this boundedness statement is the fact that the curve classes with respect to which semistability is considered might lie outside $\Pos(X)_{\mathbb R}$. This does not occur when the Picard rank of $X$ is less than three and is easier to deal with when the rank of the torsion-free sheaves is at most two. In the remaining cases we introduce a larger cone than $\Pos(X)$ consisting of ``positive'' curve classes,
$$C^+(X)\;\;:= \bigcup_{L\in\Amp(X)_\R}L^{n-2}\cdot  K^+_L(X)\;\;\; \subset\;\; N_1(X)_\R,$$ where 
$K^+_L(X):=\{ \beta\in N^1(X)_\R \ | \ \beta^2L^{n-2}>0, \ \beta L^{n-1}>0\}.$ This cone coincides with the ``positive'' component of the cone of curve classes of positive self-intersection when $\dim(X)=2$. Moreover, it naturally appears in the proof of Bogomolov's Inequality, cf.~\cite[Sect.~7.3]{Bible}. The miracle that makes the proof of Theorem~\ref{thm:boundedness_II} work is that $C^+(X)$ is still convex when $\dim(X)=3$.  

\begin{remark}
We expect the above theorem to hold more generally (that is for manifolds of any dimension), and it is natural to ask if it even holds when $K$ is allowed to be a convex polyhedral subset of the interior of the movable cone.
\end{remark}

\subsection{Linear rational chamber structure on $\mathrm{Pos}(X)_\mathbb{R}$}\label{subsect:chamberstructure_on_Pos}
Once boundedness has been established, one can study the change of semistability as the polarisation varies over a compact set in the positive cone, see \cite[Thm.~6.6]{GrebToma}.
\begin{proposition}[Chamber structure on $\mathrm{Pos}(X)_\mathbb{R}$]
 For any set of topological invariants of torsion-free sheaves on $X$ and for any compact subset $K \subset \mathrm{Pos}(X)_\mathbb{R}$, there exist finitely many linear rational walls defining a chamber structure witnessing the change in semistability in the following sense: if two elements $\alpha$ and $\beta$ in $K$ belong to the same chamber then for any torsion-free coherent sheaf $E$ with the given topological invariants, $E$ is $\alpha$-(semi)stable if and only if $E$ is $\beta$-(semi)stable. Moreover, a refined chamber structure can be constructed that additionally witnesses the change of Jordan-H\"older filtrations of semistable sheaves.
\end{proposition}
\begin{remark}
We note that the chamber structure is not necessarily effective, i.e., the set of semistable sheaves does not necessarily change when a wall is crossed, and that it depends on the chosen set $K$, i.e., enlarging $K$ to $\hat K$ might introduce new walls (which will be ineffective) in $K$.
\end{remark}

By Proposition \ref{P} our chamber structure on $\Pos(X)_{\mathbb R}$ pulls back to a locally finite chamber structure on $\Amp(X)_\mathbb{R}$ witnessing the change of slope-stability. The corresponding walls  thus obtained in $\Amp(X)_\mathbb{R}$ are given by equations that are homogeneous of degree $n-1$, so, except in the case when $\dim_\mathbb{R} N^1(X)\leq 2$, these need not be linear, see Figure 1 below. This explains the pathologies encountered in the approaches of Schmitt and Qin. More precisely, on the one hand, as we have seen, Schmitt considers segments connecting rational points in $\Amp(X)_\mathbb{R}$ as well as points on these segments where the induced notion of (slope-)stability changes. These separating points are precisely the intersection points of his segments with our walls, which clarifies the appearance of non-rational points in Schmitt's example. At the same time, the non-linearity also explains the pathologies of Qin's \emph{linear rational} chamber structure on $\Amp(X)_\mathbb{R}$, and in particular the fact that the latter cannot be locally finite in general.  

The rationality of the walls and a repeated use of Hard Lefschetz allow one to prove the following finer property of the chamber structure.
\begin{proposition}[Representing chambers by complete intersection curves]\label{prop:a_cic_in_every_chamber}
 Let  $X$ be a projective manifold of dimension $n>2$ and fix some chamber (with respect to some chosen compact set $K$ as above) $\mathcal{C} \subset \mathrm{Pos}(X)_\mathbb{R}$ of stability polarisations in $\mathrm{Pos}(X)_\mathbb{R}$. Then, there exist some ample integral classes $A$ and $B$ such that the complete intersection class 
 $A^{n-2}.B$ lies in $\mathcal{C}$.
\end{proposition}
We emphasise that the previous proposition makes a statement about every chamber, also about those of non-maximal dimension contained in one or several walls. 

As in Schmitt's example \cite[Ex.~1.1.5]{Schmitt} the Picard number of the base manifold is equal to two, Proposition~\ref{prop:a_cic_in_every_chamber} takes the following form: if $H$ defines a non-rational wall in the segment $S \subset \Amp(X)_\mathbb{R}$ under consideration, then some real multiple of $[H]^2$ is equal to $[A].[B]$ for some ample integral divisors $A$ and $B$ on $X$. Therefore, the corresponding stability condition is equivalent to the one defined by the multipolarisation $(A, B)$, at which point the non-rationality problem simply disappears!

\begin{center}
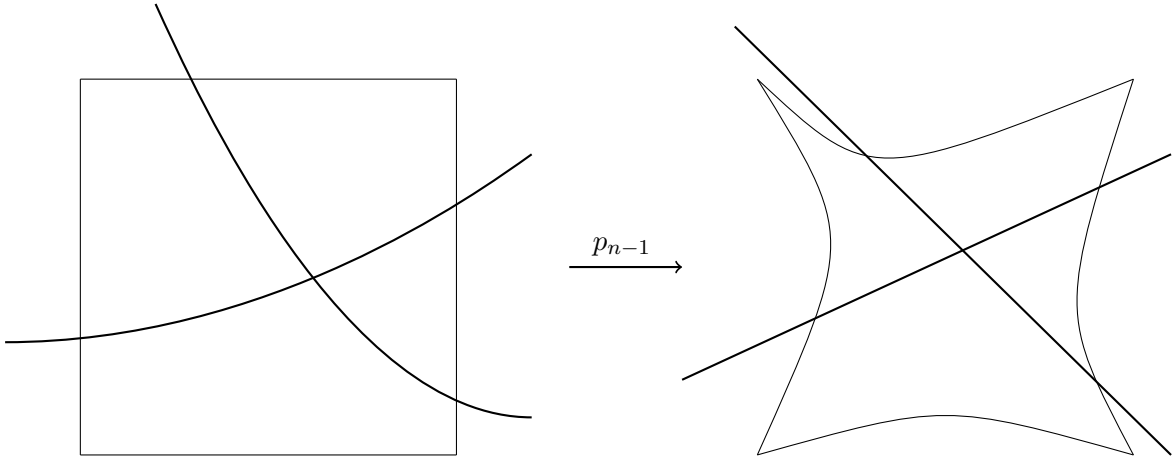
\begin{figure}[h]\label{figure1}
\begin{tikzpicture}
 \draw (-10, 0) -- (-5,0);
\draw (-5,0) -- (-5,5);
\draw (-10, 0) -- (-10, 5);
\draw (-10, 5) -- (-5,5);
\draw[thick,->] (-3.5, 2.5) node[above right]{$\;\,p_{n-1}$} -- (-2, 2.5);
 \draw (-1,0) .. controls(1.5,0.7) .. (4,0) ;
 \draw (4,0) .. controls(3,1.9) .. (4,5);
 \draw (4,5) .. controls(0.5,3.6)  .. (-1,5);
 \draw (-1,5)  .. controls(0.3,2.9) .. (-1,0);
 \draw[thick] (-1.3,5.7) -- (4.5,0);
 \draw[thick] (-2, 1) -- (4.5, 4);
\draw[thick](-4, 0.5) parabola  (-9, 6);
 \draw[thick] (-11, 1.5) parabola (-4, 4);
\end{tikzpicture}
\caption{A cross section through a compact set in $N^1(X)_\R$ and its image in $N_1(X)_\mathbb{R}$.}
\end{figure}
 \end{center}

\section{A moduli space for slope-semistable sheaves}
We have seen that the ``change of polarisation'' question (Q2), when investigated on the level of slope-functions, naturally leads to the discussion of slope-stability with respect to classes of $1$-cycles on the given projective manifold $X$. At the same time, the previous discussion shows that in order to solve question (Q1) concerning the existence of moduli spaces of slope-semistable sheaves for an arbitrary class in $\mathrm{Pos}(X)_\mathbb{R}$ it suffices to construct moduli spaces for sheaves that are slope-semistable with respect to multipolarisations $(H_1, \dots, H_{n-1})$, cf.~Proposition~\ref{prop:a_cic_in_every_chamber}. The latter situation has the huge advantage that up to some multiple the class under consideration has a geometric description: it can be represented by an actual smooth curve arising as the intersection of $n-1$ general ample divisors in the linear systems of $H_1, \dots, H_{n-1}$. 

In this section we show how to construct a moduli space for $(H_1, \dots, H_{n-1})$-slope-semistable sheaves.

\subsection{Formulation and discussion of results}
We recall that it is known that a coarse moduli space $\mathcal{M}^{\mu\text{-}s}:= \mathcal{M}^{\mu\text{-}s}_{(H_1, \dots, H_{n-1})}(\underline {c})$ for $(H_1, \dots, H_{n-1})$-slope-stable vector bundles with fixed topological invariants $\underline{c}$ exists, for example as an open separated subscheme of the moduli space of simple bundles on $X$, as constructed by Altman-Kleiman \cite{AltmanKleiman} and Kosarew-Okonek~\cite{KosarewOkonek}. It contains as a closed subspace the moduli space $\mathcal{M}^{\mu\text{-}s}(\Lambda)$ of slope-stable bundles with fixed determinant line bundle $\Lambda$. In \cite{GrebToma}, the following main result is proven, providing a natural generalisation of Jun Li's algebro-geometric construction of the Donaldson-Uhlenbeck compactification to higher dimensions:

\begin{theorem}[A projective moduli space for slope-semistable sheaves]\label{thm:slopemoduli}
Let $X$ be a projective manifold of dimension $n \geq 2$, $H_1,$ $...,$ $H_{n-1}$ be ample divisors, 
$c_i \in H^{2i}\bigl(X, \mathbb{Z}\bigr)$ for $1\le i\le n$,  $r$ a positive integer, $\underline{c} \in  K(X)_{num}$ a class with rank $r$ and Chern classes $c_j(c) = c_j$, 
and $\Lambda$ a line bundle on $X$ with $c_1(\Lambda)= c_1 \in H^2 (X, \mathbb{Z})$. Then, there exists a ``modular'' semi-normal projective compactification $\mathcal{M}^{\mu\text{-}ss}(\Lambda) = \mathcal{M}^{\mu\text{-}ss}(\underline{c}, \Lambda)$ of the semi-normalisation $(\mathcal{M}^{\mu\text{-}s}(\Lambda))^{w\nu}$.
\end{theorem}
To explain the statement in more detail, we remark the following points:

1) A reduced complex-analytic space $X$ is called \emph{semi-normal} or \emph{weakly normal}, if for every open subset $U \subset X$ the restriction map 
\[\mathscr{O}_X(U) \to \mathscr{C}_X(U) \cap \mathscr{O}_X(U\cap X_{reg})\]
is surjective, i.e., if every continuous function that is holomorphic on the regular part of $U$ is already holomorphic on $U$. Every complex space has a functorial semi-normalisation $X^{w\nu} \to X$, i.e., maps of spaces lift to maps of semi-normalisations, see for example \cite[\S 15]{RemmertSCV_VII}. Moreover, the semi-normalisation map is a homeomorphism of the underlying topological spaces, i.e., the complex spaces $X$ and $X^{w\nu}$ differ only in their sheaf of holomorphic functions, which is why the complex structure induced by the semi-normalisation is also called \emph{maximal}.

In other words, the semi-normalisation contains all the topological information about the moduli space, which is important when aiming for the definition and computation of invariants of $X$ via topological invariants of $M^{\mu\text{-}ss}$, cf.~the theory of Donaldson invariants \cite{DK90}. We will see later where the semi-normality assumption is used in our proof of Theorem~\ref{thm:slopemoduli}, at this point we remark that Chow varieties of $2$-codimensional cycles in $X$ arise as special cases of our construction, cf.~Example~\ref{ex:moduliequalsChow} below. These are classically known to have natural scheme structures in dimension two, whereas the higher-dimensional situation is much more subtle, and in general, they are only endowed with the structure of a semi-normal variety, see for example \cite[Chap.~I, Thm.~3.21]{KollarRatCurves}. In this sense, it is not surprising that our construction only yields the maximal complex structure on the topological space underlying $M^{\mu\text{-}ss}(\Lambda)$, and not a full scheme structure containing finer information about infinitesimal families of semistable sheaves.

2) The compactification is \emph{modular} in the following sense: the variety $\mathcal{M}^{\mu\text{-}ss}(\Lambda)$ is naturally polarised by an ample line bundle $\mathscr{O}_{\mathcal{M}^{\mu\text{-}ss}(\Lambda)}(1)$, and there exists a natural number $N$ having the following property: for every flat family $\mathscr{F}$ of $(H_1, \dots, H_{n-1})$-semistable sheaves with the chosen topological invariants and determinant $\Lambda$ parametrised by a semi-normal scheme $B$, there exists a 
 classifying morphism $\phi_{\mathscr{F}}: B \to \mathcal{M}^{\mu\text{-}ss}(\Lambda)$ such that
\[\phi_{\mathscr{F}}^* \bigl(\mathscr{O}_{\mathcal{M}^{\mu\text{-}ss}(\Lambda)}(1)\bigr) = \lambda_\mathscr{F}^{\otimes N},\]
where $\lambda_\mathscr{F}$ is a certain determinant line bundle on the base $B$ of the family $\mathscr{F}$. The triple $(\mathcal{M}^{\mu\text{-}ss}(\Lambda), \mathscr{O}_{\mathcal{M}^{\mu\text{-}ss}(\Lambda)}(1), N)$  is uniquely determined up to isomorphism by this and a further, slightly technical, property, see the original paper \cite{GrebToma} for details.

3) Fixing the determinant is not a serious restriction, we expect that essentially the same proof will lead to a moduli space that is projective over the component of the Picard variety $\mathrm{Pic}^{c_1}(X)$ parametrising line bundles with first Chern class $c_1$ on $X$.

\subsection{Idea of the proof}
One of the main ideas of the proof is to use the following improvement of Bogomolov's effective restriction theorem \cite[Sect.~7.3]{Bible} due to Langer \cite[Thm.~5.2 and Cor.~5.4]{Langer}:
\begin{theorem}[Semistable Restriction Theorem]\label{prop:slopeMR} If the torsion-free coherent sheaf $E$ is (semi)stable with respect to the multipolarisation 
$(H_1, ..., H_{n-1})$, then there is a positive threshold $k_0\in \N^{>0}$ depending only on the topological type of $E$ such that the following holds: if $k\ge k_0$ and if $D\in |kH_1|$ is any smooth divisor such that 
\begin{enumerate}
 \item $E|_D$ is torsion-free, and 
 \item the restriction $F_j|_D$ of any Jordan-H\"older factor $F_j$ of $E$ is torsion-free,
\end{enumerate}
  then the restriction $E|_D$ is (semi)stable with respect to $(H_2|_D, ..., H_{n-1}|_D)$.
\end{theorem}
As in the traditional proof of the construction of the moduli space of Gieseker-semistable sheaves explained for example in \cite{Bible}, using boundedness we parametrise slope-semistable sheaves by a locally closed subscheme $R^{\mu\text{-}ss}$ of a suitable Quot-scheme, whose universal family we denote by $\mathscr{F}$. Isomorphism classes of slope-semistable sheaves correspond to orbits of a special linear group $G$ in $R^{\mu\text{-}ss}$. We then consider the determinant line bundle $\lambda_\mathscr{F}$ on $R^{\mu\text{-}ss}$ and aim to show that it is generated by $G$-invariant global sections, at least after passing to the semi-normalisation. 

Possibly after replacing the divisors $H_1, \dots H_{n-1}$ by high enough multiples, by Theorem~\ref{prop:slopeMR} above for every sheaf $E$ contained in the family $R^{\mu\text{-}ss}$ there exists a smooth complete intersection curve $C$ for the linear systems $|H_1|, \dots, |H_{n-1}|$ such that the restricted sheaf $E|_C$ is semistable (note that in order to define semistability on $C$ no further choice of polarisation is necessary, cf.~ the discussion at the beginning of the Introduction). Moreover, using the fact that on the curve $C$ the moduli space of semistable vector bundles can be constructed using Geometric Invariant Theory (whose definition of semistability requires certain invariant sections to exist) and a determinant line bundle computation comparing $\lambda_\mathscr{F}$ with the determinant line bundle of the restricted family $\mathscr{F}|_{C \times R^{\mu\text{-}ss}}$, one sees that at least over the locus in $R^{\mu\text{-}ss}$ where the restricted family stays flat there exists a $G$-invariant section in some power of $\lambda_\mathscr{F}$ that does not vanish on the orbit corresponding to $E$ in $R^{\mu\text{-}ss}$. This is the point at which our restriction to semi-normal parameter spaces comes into play: in order to overcome the difficulty created by the lacking flatness of restricted families, we pass to the semi-normalisation and show that sections in powers of $\lambda_\mathscr{F}$ extend continuously, and hence holomorphically, over the non-flat locus.

 The moduli space $\mathcal{M}^{\mu\text{-}ss}(\Lambda)$ then arises as the Proj-scheme of a ring of $G$-invariant sections of powers of $\lambda_\mathscr{F}$ over the semi-normalisation of $R^{\mu\text{-}ss}$. Afterwards, the universal properties are established using the $G$-equivariant geometry of $R^{\mu\text{-}ss}$. 
\subsection{Geometry of the moduli space}\label{subsect:geometry_of_slopemoduli}
The fact that $\mathcal{M}^{\mu\text{-}ss}(\Lambda)$ is a compactification of the moduli space of slope-stable bundles can be generalised to the following more precise statements about the geometry of $\mathcal{M}^{\mu\text{-}ss}(\Lambda)$, which should be compared with the two-dimensional situation, as described for example in \cite[Thm.~8.2.11 and Rem.~8.2.17]{Bible}.

Let $E$ be a $(H_1, \dots, H_{n-1})$-slope-semistable torsion-free sheaf on $X$. Then, $E$ has a Jordan-H\"older filtration, cf.~\cite[Thm.~1.6.7]{Bible}, such that the associated graded sheaf $gr^\mu E$ is a direct sum of torsion-free stable sheaves. The double-dual $E^\sharp := (gr^\mu E)^{**}$ sits in an exact sequence
\[0 \to gr^\mu E \to E^\sharp \to Q_E \to 0,\]
where $Q_E$ is a torsion sheaf supported in codimension two or higher. Using a natural ``Quot-to-Chow''-morphism, from $Q_E$ we can compute a support Chow-cycle of dimension $n-2$, which we denote by $C_E$. Let now $E$ and $F$ be two semistable sheaves. Assume that either $E^\sharp \not \cong F^\sharp$ or $C_E \neq C_F$. Then, $E$ and $F$ give rise to different points in $\mathcal{M}^{\mu\text{-}ss}(\Lambda)$. If the natural Quot-to-Chow-morphism used to define $Q_E$ has connected fibres, the moduli space exactly parametrises pairs $(E^\sharp, C_E)$ as above. While this connectedness holds in dimension two due to a result of Ellingsrud-Lehn \cite{EllLehn}, the same statement fails in higher dimensions, as can be seen by looking at Example~\ref{ex:Hilb-to-Chow} below. 

On the other hand, if $X$ is a threefold (for example $\mathbb{P}^3$) and $E$ is some fixed stable bundle on $X$, twisting with the ideal sheaf of varying points in $X$ yields a family of semistable sheaves that will not be separated by the corresponding moduli space, as the determinant line bundle used to define $\mathcal{M}^{\mu\text{-}ss}(\Lambda)$ is insensitive to the change in codimension three. The necessary computation is similar to \cite[Ex.~8.1.3(i) and Ex.~8.1.8(ii)]{Bible}.

\begin{example}[Non-connected fibres of Hilb-to-Chow morphisms]\label{ex:Hilb-to-Chow}
Let $\dim X\ge 4$, let $H$ be any ample polarisation, and let $Z$ be some subscheme of $X$ of codimension at least three. Denote by $P:=P_H(\mathscr{I}_Z)$ the Hilbert polynomial of the ideal sheaf of $Z$ with respect to $H$. Then, the corresponding Gieseker moduli space $\mathcal{M}_H(P)$ of semistable sheaves with Hilbert polynomial $P$ coincides with the Hilbert scheme $\Hilb_{P(\mathscr{O}_X)-P}(X)$.  For each such ideal sheaf the associated codimension-two cycle is the zero cycle, and thus the whole $\mathcal{M}_H=\Hilb_{P(\mathscr{O}_X)-P}(X)^{red}$ is naturally mapped to a single point in $\mathrm{Chow}_{n-2}(X)$. This map factors through the  corresponding moduli space of slope-semistable sheaves $\mathcal{M}^{\mu\text{-}ss} (\mathscr{O}_X)$, which consists of precisely $m$ simple points, where $m$ is the number of connected components of $\Hilb_{P(\mathscr{O}_X)-P}(X)$. Thus, $\mathcal{M}^{\mu\text{-}ss}(\mathscr{O}_X)$ appears here as the Stein factorisation of a Hilb-to-Chow morphism.  Looking at hypersurfaces in $\mathbb{P}^n$ containing finitely many lines we obtain explicit examples in which $m > 1$, i.e., in which the fibres of the Hilb-to-Chow morphism are not connected.
\end{example}

\begin{example}[Chow varieties]\label{ex:moduliequalsChow}
 As a variation of the previous example one may consider the case of Hilbert schemes of codimension-two subschemes $Z \subset X$. If the corresponding Hilb-to-Chow morphism has connected fibres, $\mathcal{M}^{\mu\text{-}ss}(\mathscr{O}_X)$ coincides with $\mathrm{Chow}_{n-2}(X)$.
\end{example}

\section{Variation of Gieseker moduli spaces}
We have seen that one way to deal with the non-linearity of the chamber structure in $\Amp(X)_{\mathbb R}$ that appears in the change-of-polarisation problem is to pass to a different space (namely $\Pos(X)_{\mathbb R}\subset N_{1}(X)_{\mathbb R}$) in which the problem is more tractable. While for slope-stability this approach does not destroy any information, from the point of view of Gieseker-stability it ignores everything but the first non-trivial coefficient of the (reduced) Hilbert polynomial. In this section, will see that the variation problem for Gieseker-stability can be tackled by moving to a space of stability conditions that depend on several polarisations.

\subsection{Gieseker and multi-Gieseker stability}\label{subsect:MultiVariation}
 Recall a torsion-free coherent sheaf $E$ is said to be \emph{Gieseker-(semi)stable} (with respect to a given ample line bundle $L$) if for all non-trivial coherent proper subsheaves $F$ of $E$  it holds that
\begin{equation*}
 \frac{\chi(F\otimes L^n)}{\rank(F)} < (\le) \frac{\chi(E\otimes L^n)}{\rank(E)}\quad \text{ for all } n\gg 0.\label{eq:defGiesekerstability}
 \end{equation*}
This extends to a definition involving several ample line bundles as follows:  let $L_{1},\ldots,L_{j_{0}}$ be ample line bundles on $X$ and $B_{1},\ldots,B_{j_{0}}$ be a further collection of line bundles, and suppose   $\sigma=(\sigma_1,\ldots,\sigma_{j_0})$ is a non-zero vector of non-negative real numbers.  We shall say a torsion-free coherent sheaf $E$ on $X$ is \emph{multi-Gieseker-(semi)stable} with respect to this data if for all non-trivial coherent proper subsheaves $F$ of $E$ it holds that
\begin{equation*}
 \frac{\sum_{j} \sigma_j \chi(F\otimes B_{j}\otimes L_{j}^n)}{\rank(F)} < (\le)  \frac{\sum_{j} \sigma_j \chi(E\otimes B_{j}\otimes L_{j}^n)}{\rank(E)}\quad \text{ for all }n\gg 0.\label{eq:defmultiGiesekerstability}
 \end{equation*}
 We stress that we allow the $\sigma_{j}$ to be irrational, but for the above Euler characteristic to make sense we require $L_{j}$ and $B_{j}$ to be genuine (i.e.\ integral) line bundles.  Since we will be thinking of the $L_{j}$ and $B_{j}$ as fixed, we will sometimes refer to this as \emph{$\sigma$-stability}.  Thus, if all the $B_{j}$ are trivial, multi-Gieseker stability is simply a convex combination of Gieseker-stability for each of the $L_{j}$ individually, and thus naturally interpolates between the individual notions of Gieseker-stability as the $\sigma_{j}$ vary. Moreover, there is a close connection to the notion of slope-stability as discussed in previous sections:

\begin{lemma}[Comparison between slope and multi-Gieseker-stability]\label{lemma:slope_Gieseker_relation}
With notation and setup as above, set
$$ \gamma = \sum\nolimits_j \sigma_j c_1(L_j)^{\dim X-1} \in N_1(X)_\mathbb{R}.$$
Then, for any torsion-free coherent sheaf $E$ the following implications hold
\begin{center}
slope-stable with respect to $\gamma$ $\Rightarrow$ stable with respect to $\sigma$ $\Rightarrow$ semistable with respect to $\sigma$ $\Rightarrow$ slope-semistable with respect to $\gamma$.
\end{center}
\end{lemma}

Given this, we now state the main results from \cite{GRTI,GRTII}, which for simplicity we do under the additional assumption that $X$ has dimension 3 (in fact all these results hold in some form for singular varieties of all dimension under some additional hypotheses). We fix the topological type (i.e., Chern classes) of the sheaves in question, which we tacitly assume has been done from now on.  

\begin{theorem}[Existence of projective moduli spaces]\label{thm:multiGieseker_moduli_exist} 
There exists a projective moduli space $\mathcal M_{\sigma}$ of multi-Gieseker-semistable sheaves.
\end{theorem}
As is clear, the moduli space $\mathcal M_L$ of Gieseker-semistable sheaves with respect to a single ample line bundle $L$ is a special case of this construction (simply taking all but one of the $\sigma_{j}$ to be zero).   Moreover, just as for $\mathcal M_L$, the moduli spaces $\mathcal M_{\sigma}$ contain an open set parameterising stable sheaves, and the points on the boundary correspond to $S$-equivalence classes of sheaves.  

In the spirit of this paper, our interest in $\mathcal M_{\sigma}$ really comes from how it changes as $\sigma$ varies in $\mathbb R^{j_{0}}_{{\ge 0}}\setminus\{0\}$. The following result, which is proven in \cite[Sect.~4]{GRTI}, shows that by passing from the ample cone to $\sigma$-space, the non-linearity and non-rationality issues encountered before disappear. This in turn should be compared with the strategy employed in the case of slope-semistability described in Section~\ref{subsect:chamberstructure_on_Pos}.

\begin{theorem}[Chamber structure]\label{thm:multiGieseker_chamber_structure}
Let $X$ be a smooth projective threefold and let $L_{1},\ldots,L_{j_{0}}$ be fixed ample line bundles on $X$.  Then, the set $\mathbb R^{j_{0}}_{\ge 0}\setminus\{0\}$ is cut into chambers by a finite number of linear rational walls, such that the notion of semistability as well as the induced Jordan-H\"older filtrations are unchanged as $\sigma$ varies in the interior of a single chamber. \end{theorem}

Using this, we can compare the moduli spaces of Gieseker-stable sheaves with respect to different ample polarisations.  We first state a preliminary version of the corresponding result, see \cite[Thm.~12.1]{GRTI}.

\begin{theorem}[Variation of Gieseker moduli spaces on threefolds]\label{thm:variation_on_threefolds}
Let $X$ be a smooth projective threefold and let $L_{1},L_{2}$ be ample line bundles on $X$.    Then, the moduli spaces $\mathcal{M}_{L_1}$ and $\mathcal{M}_{L_2}$ of sheaves that are Gieseker-semistable with respect to $L_1$ and $L_2$, respectively, are related by a finite number of Thaddeus-flips.
\end{theorem}

Here, by a \emph{Thaddeus-flip} we mean a transformation occurring as a change of Geometric Invariant Theory (GIT) stability on a fixed ``master space''.  More precisely, we say two schemes $X^+$ and $X^-$ are related by a Thaddeus-flip if there exists a quasi-projective scheme $R$ with an action of a reductive group $G$ and stability parameters $\sigma_+,\sigma_-,\sigma$ such that there exists a diagram of the form
\[\begin{xymatrix}{
  X^+=R^{ss,\sigma_+}\hq G \ar[rd]_{\psi_+}& & \ar[ld]^{\psi_-} R^{ss,\sigma_-}\hq G=X^-\\
   &  R^{ss, \sigma}\hq G, &  
}
\end{xymatrix}
\]
where  $R^{ss,\sigma}$ denotes the set of points that are GIT-semistable with respect to $\sigma$, and the morphisms $\psi_{\pm}$ are induced by inclusions $R^{ss,\sigma_+}\subset R^{ss,\sigma}\supset R^{ss,\sigma_-}$.    We emphasise that a Thaddeus-flip is not necessarily a flip in the sense of birational geometry, since a priori even if all the spaces involved are non-empty, this transformation could be a divisorial contraction or contract an irreducible component. However, from the theory of Variation of GIT due to Thaddeus \cite{Thaddeus} and Dolgachev-Hu \cite{DolgachevHu} it will consist of a sequence of birational flips under certain circumstances. 

\begin{remark}
We note that while such a result was expected to be true by many experts, as explained in the introduction the methods of Matsuki and Wentworth \cite{MatsukiWentworth} employed in the surface case do not generalise to higher dimensions.   Moreover, the most natural alternative approaches run into problems, as we will indicate in Section~\ref{subsect:quivervariation} below.
\end{remark}

Having established such a variation result, we can be more ambitious and ask for information as to the kind of Thaddeus-flips between such moduli spaces. More specifically, we ask if they can be given a interpretation as moduli spaces of sheaves of some kind.   Although we have not answered this in full generality, we develop in \cite{GRTII} a framework that strongly suggests this is always the case.  In particular, we obtain the following results:

\begin{theorem}[Mumford-Thaddeus Principle through moduli spaces of sheaves]
Let $L_{1},L_{1}$ be ample. Suppose also that either
\begin{enumerate}
\item $X$ has dimension 2, or
\item $X$ has dimension 3 and $L_{1},L_{2}$ are ``separated by a single wall of the first kind''.
\end{enumerate}
Then, the moduli spaces $\mathcal M_{L_{1}}$ and $\mathcal M_{L_{2}}$ are related by a finite number of Thaddeus-flips through moduli spaces of multi-Gieseker semistable sheaves.
\end{theorem}

Part (1) is precisely what was proven by Matsuki and Wentworth in \cite{MatsukiWentworth}.  The definition of being ``separated by a single wall of the first kind'' is a certain genericity assumption on $L_{1}$ and $L_{2}$, for which we refer the reader to \cite[Def.~7.5]{GRTII} (and which is likely not to be necessary).  

\begin{remark}[Compactified Jacobians]\label{rem:Alexeev}
 Another setup in which the results of \cite{GRTII} can be applied is the construction and comparison of compactified Jacobians, see \cite{AlexeevCompactifiedJacobians}: Let $C$ be a reduced nodal curve with irreducible components $C_i$, and let $L$ be an ample line bundle on $C$. For any coherent sheaf $F$ set $\deg F := \chi(F) - \chi(\mathscr{O}_C)$. As $\dim C =1$, slope- and Gieseker semistabilty with respect to $L$ agree, and we will hence just talk about $L$-semistability. Then, the \emph{compactified Jacobian} $\mathrm{Jac}_{d,L}(C)$ is defined to be the moduli space of $L$-semistable sheaves of degree $d$ that have rank $1$ at every generic point of $C$. In this situation, one obtains that as $L$ varies the $\mathrm{Jac}_{d,L}(C)$ are related by Thaddeus flips through moduli spaces of the form $\mathrm{Jac}_{d, L'}(C)$, and that these flips are induced by a fixed master space. Note however that a precise description of this change of moduli spaces has been given in terms of a combinatorial construction of compactified Jacobians, cf.~\cite[Rem.~2.10]{AlexeevCompactifiedJacobians}.
\end{remark}

\subsection{From sheaves to representation of quivers}\label{subsect:quivervariation}

We now sketch the main ideas involved in the construction of the moduli space of multi-Gieseker semistable sheaves and the derivation of the variation results. The general philosophy is adapted from the strategy put forward and executed by \'Alvarez-C\'onsul and King in \cite{ConsulKing}; see also the survey \cite{ConsulKingSurvey}.

Let us first recall the classical construction of the Gieseker moduli space, see \cite[Chap.~4]{Bible}:  For any coherent sheaf $E$ living in a bounded family we can choose $n$ sufficiently large and uniformly over the family so that $E\otimes L^n$ is globally generated, i.e., the evaluation map 
$ H^0(E\otimes L^n) \to E\otimes L^n$
is surjective.  Thus, choosing an isomorphism $H^0(E\otimes L^n)\simeq V$, where $V$ is a fixed vector space of the appropriate dimension, we can thus consider $E \otimes L^n$ as a point in the Quot scheme of the trivial bundle with fibre $V$.  Expanding slightly, letting $H:=H^0(L^{m-n})$,
we have that for $m$ sufficiently large the natural multiplication
\begin{equation}\label{eq:multiplication}
 V \otimes H \simeq H^0(E\otimes L^n) \otimes H  \to H^0(E\otimes L^m)
 \end{equation}
is surjective, and thus gives a point in a Grassmannian of $V\otimes H$.  The different choices of isomorphism correspond to the orbits of this point under the natural $GL(V)$-action,  thus the moduli space desired is the quotient with respect to $GL(V)$. This quotient can be constructed using GIT, and it is at this stage that the stability condition enters. 

So, while the Gieseker moduli spaces with respect to several given ample polarisations are all constructed by GIT, each one of them is the quotient of a different Grassmannian by a different group action. In the situation where one wants to compare these moduli spaces, one could of course consider the product of the Grassmannians and some ``diagonal'' embedding, but then the natural polarisations induced on the product are not ample or even big, and it is not clear how to proceed from here. 

The main insight of \'Alvarez-C\'onsul and King is the observation that it is possible to delay the point at which one picks the rigidifying isomorphism in \eqref{eq:multiplication}. We again refer the reader to the original paper \cite{ConsulKing} and especially to the survey \cite{ConsulKingSurvey} for a thorough exposition of this approach, here we now shall discuss our generalisation of their construction. For simplicity, we restrict to the case of two polarisations; so, suppose $L_{1}$ and $L_{2}$ are very ample line bundles, and $B_{1}=B_{2}=\mathscr{O}_{X}$ are trivial.  For $i,j=1,2$, set $H_{ij}:= H^{0}(L_{i}^{-n}\otimes  L_{j}^{m})$.  Then, to any coherent torsion-free sheaf $E$ and integers $m>n$ we associate the following diagram of linear maps, which we denote by $\Phi(E)$:
\[
\begin{xymatrix}{
   H^0(E\otimes L_1^n)  \ar[rr]|{H_{11}} \ar[rrd]|<<<<<{H_{12}} & & H^0(E\otimes L_1^m) \\
 H^0(E\otimes L_2^n)   \ar[rr]|{H_{22}} \ar[rru]|<<<<<{H_{21}}  & &  H^0(E\otimes L_2^m).}
  \end{xymatrix}
\]
Here, the maps are the natural multiplication maps, and the decoration $H_{ij}$ indicates that the multiplication goes from the tensor-product with the source to the target; so for example, the top row is the multiplication 
\begin{equation}
H^0(E\otimes L_1^n) \otimes H_{11}  = H^0(E\otimes L_1^n)\otimes H^0(L_1^{m-n})\to H^0(E\otimes L_1^m).\label{eq:toprow}
\end{equation}

The first point we make is that it is possible to choose $m\gg n\gg 0$ sufficiently large so that this diagram ``recovers'' the sheaf $E$ , up to isomorphism (in fact one can recover $E$ from just the top row \eqref{eq:toprow} of the diagram, as is done in the classical construction of the Quot scheme). More conceptually, we can think of $\Phi(E)$ as a representation of the quiver $Q$ that has 4 vertices and $\dim(H_{ij})$ arrows between the appropriate vertices.  As such, the association $E\mapsto\Phi(E)$ is a functor that allows us to embedd the category of (suitably regular) sheaves into the category of representations of $Q$.

\begin{theorem}[Embedding regular sheaves into a category of quiver representations]\label{thm:categoryembedding}
For $m\gg n$ the following holds:  for any sheaves $E,F$ (of given topological type) that are $n$-regular with respect to both $L_{1}$ and $L_{2}$, the map
$$\Phi_{*}\colon \Hom(E,F) \to \Hom(\Phi(E),\Phi(F))$$ 
is an isomorphism.  
\end{theorem}

This in particular implies that $E\mapsto \Phi(E)$ is injective on (isomorphism classes of) objects, and thus $E$ can be recovered from $\Phi(E)$.

\begin{remark} In the setup of Theorem~\ref{thm:categoryembedding}, the regularity assumption and the assumption $m\gg n$ imply, among other things, that $E\otimes L_{i}^{n}$ and $E\otimes L_{j}^{m}$ are globally generated, and that the maps $H^{0}(E\otimes L_{i}^{n})\otimes H_{ij}\to H^{0}(E\otimes L_{j}^{m})$ are surjective.  From this, injectivity of $\Phi_{*}$ is immediate, for if $\Phi_{*}(f)=0$ for some $f\colon E\to F$ then certainly $f_{*}\colon H^{0}(E\otimes L_{1}^{n})\to H^{0}(F\otimes L_{1}^{n})$ is zero, and thus $f=0$. Surjectivity can be proved with a lengthy diagram chase, which again uses the regularity assumption on $E$ and $F$, and depends in a crucial way on the fact that the quiver $Q$ contains diagonal arrows relating the two rows; see \cite[Sect.~5]{GRTI} for an alternative proof using pushout diagrams.
\end{remark}

Using the embedding $\Phi$, we can now employ earlier work of King \cite{King} that uses Geometric Invariant Theory to produce moduli spaces of representations of a given quiver.  To do so, King introduces a notion of stability for quiver representations, which for our purpose we present as follows: A representation $R$ of $Q$ is given by a diagram of linear maps of the form
\[\begin{xymatrix}{
   V_{1} \ar[rr]|{H_{11}}\ar[rrd]|<<<<<{H_{12}} & & W_{1} \\
 V_{2}  \ar[rr]|{H_{22}} \ar[rru]|<<<<<{H_{21}}  & &  W_{2}
}
  \end{xymatrix}
\]
for some finite dimensional vector spaces $V_{j}$ and $W_{j}$ (which to avoid pathologies we assume to be non-trivial).  Given non-negative rational numbers $\sigma_{1},\sigma_{2}$ we define the \emph{slope} of $R$ to be the quantity 
$$ \mu(R):=\frac{ \sigma_{1}\dim V_{1} + \sigma_{2}\dim V_{2}}{ \sigma_{1}\dim W_{1} + \sigma_{2}\dim W_{2}},$$
and $R$ is said to be (semi)stable if $\mu(R')< (\le)\mu(R)$ for all proper non-trivial subrepresentations $R'$ of $R$. 

 We note in passing that \cite{ConsulKing} uses a smaller quiver, the so-called \emph{Kronecker quiver} $$\bullet \overset{H}{\longrightarrow} \bullet$$ to construct the Gieseker moduli space. What makes our quiver more interesting is the existence of a non-trivial space $\{(\sigma_1, \sigma_2) \mid \sigma_j \in \mathbb{R}_{\geq 0}\} $ of stability conditions that one can vary (whereas for the Kronecker-quiver there is essentially only one), and it is this aspect that allows us to study the change-of-polarisation problem using this machinery. 

Now, King proved that there exists a projective moduli space $\mathcal Q$ of semistable representations of $Q$ (where the $V_{i}$ and $W_{i}$ are required to have the correct dimension determined by the topological type of the sheaves in question). Then, the moduli space of sheaves $\mathcal M_{\sigma}$ can be constructed as a subscheme of $\mathcal Q$, so long as the stability notions for sheaves and representations match up, which is the content of the next statement, see \cite[Thm.~8.1]{GRTI}:

\begin{theorem}[Comparison of semistability]\label{thm:semistabilitycomparison}
For all integers $m\gg n\gg p\gg 0$ the following holds:  any torsion-free sheaf $E$ on $X$ (of given topological type) is multi-Gieseker-semistable (with respect to $L_{1},L_{2},\sigma_{1},\sigma_{2}$) if and only if (a) it $p$-regular with respect to $L_{1}$ and $L_{2}$, and (b) the representation $\Phi(E)$ is semistable. In addition, if each $\sigma_j$ is strictly positive, $\Phi(E)$ is compatible with the natural Jordan-H\"older filtrations on sheaves and quiver representations, respectively.
\end{theorem}

\begin{remark}
Observe that in the final statement of Theorem \ref{thm:semistabilitycomparison} we require the $\sigma_j$ to be strictly positive.  As a consequence, it is not possible to set all but one of them to be zero in order to directly recover Gieseker-stability with respect to $L_j$ in this setup. Hence, extra steps are needed to prove Theorem~\ref{thm:variation_on_threefolds} as an application of Theorem~\ref{thm:semistabilitycomparison}. We plan to revisit this issue in the future and to show that these extra steps can be avoided.
\end{remark}

The essence of the idea of the proof can be seen as follows.  Recall that by definition a torsion-free sheaf $E$ is multi-Gieseker-semistable if and only if for all non-trivial proper subsheaves $F$ of intermediate rank we have
\begin{equation*}
 \frac{\sum_{j} \sigma_j \chi(F\otimes L_{j}^n)}{\rank(F)} \le  \frac{\sum_{j} \sigma_j \chi(E\otimes L_{j}^n)}{\rank(E)}\quad \text{ for all }n\gg 0.
 \end{equation*}
Using that the rank is the leading order term of the Hilbert-polynomial and vanishing of higher cohomology, one sees easily that this inequality is equivalent to
\begin{equation*}
 \frac{\sum_{j} \sigma_j h^{0}(F\otimes L_{j}^n)}{\sum_{j} \sigma_j h^{0}(F\otimes L_{j}^m)} \le  \frac{\sum_{j} \sigma_j h^{0}(E\otimes L_{j}^n)}{\sum_{j} \sigma_j h^{0}(E\otimes L_{j}^m)}\quad \text{ for all }m\gg n\gg 0.
 \end{equation*}
But this precisely says that $\mu(\Phi(F))\le \mu(\Phi(E))$.    Thus, what remains is (a) to show that one can take $m,n$ uniformly over all sheaves of interest and (b) that to test for stability of $\Phi(E)$ it is sufficient to consider only subrepresentations of the form $\Phi(F)$ for some subsheaf $F\subset E$.  Both points are rather technical to establish, based on a tailored version of the Le Potier-Simspon estimate, and we refer the reader to \cite[Sects.~7 and 8]{GRTI} for details.

To sum up the discussion: those representations of $Q$ that are of the form $\Phi(E)$ for some sufficiently regular semistable sheaf $E$ form a locally closed subscheme of the appropriate space $\mathcal{R}$ of representations of $Q$, which in addition is saturated with respect to the relevant GIT-quotient by Theorem~\ref{thm:semistabilitycomparison}.  So, one concludes that $\mathcal M_{\sigma}$ exists as a quasi-projective scheme, which in fact is projective by an appropriate version of Langton's Theorem. 

The existence of Thaddeus-flips between two such moduli spaces now comes from general theory, since we have used \emph{the same embedding} (Theorem \ref{thm:categoryembedding}) \emph{for all stability parameters} $\sigma$. As all the moduli spaces involved are constructed via GIT from $\mathcal{R}$, general statements from variation of GIT \cite{Thaddeus, DolgachevHu} apply to our fixed master space to establish the claim of Theorem~\ref{thm:variation_on_threefolds}.

\subsection{Semistability with respect to K\"ahler polarisations}
As we have seen in the introduction, the change-of-polarisation problem naturally leads to question (Q3) concerning the existence of moduli spaces of sheaves that are Gieseker-semistable with respect to a non-rational class in $\mathrm{Amp}(X)_\mathbb{R}$. Using the machinery of multi-Gieseker moduli spaces discussed above, we are able to give a positive answer to question (Q3) in \cite[Thm.~11.6]{GRTI}:

 \begin{theorem}[Projective moduli spaces for $\omega$-semistable sheaves]\label{thm:kaehlermodulithreefolds}
  Let $\omega\in N^1(X)_{\mathbb R}$ be a real ample class on a smooth projective threefold $X$. Then, there exists a projective moduli space $\mathcal M_{\omega}$ of torsion-free sheaves of fixed topological type that are semistable with respect to $\omega$.  This moduli space contains an open set consisting of points representing isomorphism classes of stable sheaves, while points on the boundary correspond to $S$-equivalence classes of strictly semistable sheaves.  
\end{theorem}
\begin{remark}
It is likely that the assumption that $X$ has dimension three is not really necessary. For partial results in dimensions greater than three, the reader is referred to \cite[Sect.~5]{GRTII}. We emphasise that the above moduli spaces are \emph{projective} despite us using a real class to define the stability condition, contrary to the expectation expressed for example in \cite[p.~217, after Main Theorem]{Schmitt}. Moreover, as Schmitt's example \cite[Ex.~1.1.5]{Schmitt} shows, in general these moduli spaces are new in the sense that they are not at the same time moduli spaces of sheaves that are Gieseker-semistable with respect to some ample class on $X$. We note that on the other hand the existence problem for moduli spaces of Gieseker-semistable sheaves on arbitrary compact K\"ahler manifolds is wide open and that the techniques used here rely in a crucial way on the projectivity of $X$.
\end{remark}

The proof of Theorem~\ref{thm:kaehlermodulithreefolds} consists of two steps: first, using Proposition~\ref{P} we show that for every class $[\omega] \in \mathrm{Amp}(X)_\mathrm{R}$ there exist line bundles $L_{1},\ldots,L_{j_{0}}$ and (possibly irrational) positive numbers $\sigma_{1},\ldots,\sigma_{j_{0}}$ such that the reduced ``Hilbert polynomial'' $$p_E(m) = \frac{1}{\rank(E)}\int\nolimits_X ch(E) e^{m\omega} \Todd(X)$$ with respect to $[\omega]$ is equal to the reduced multi-Hilbert polynomial with respect to the data $L_{j},\sigma_{j}$. We then observe that since the chamber structure for multi-Gieseker semistability on $\mathbb{R}_{\geq 0}^{j_0}$ is given by linear \emph{rational} walls, see Theorem~\ref{thm:multiGieseker_chamber_structure}, we can perturb the $\sigma_{j}$ to be rational without changing the notion of stability.  Then, the moduli space $\mathcal M_{\sigma}$ constructed in the previous section is the moduli space for $[\omega]$-semistable sheaves.

\section{The multi-Gieseker-to-Uhlenbeck morphism}
In this final section, we establish a connection between the two moduli problems considered above by proving the existence of a multi-Gieseker-to-Uhlenbeck morphism. 

Let $X$ be a projective manifold of dimension $n$. Fixing ample line bundles $L_{1},\ldots,L_{j_{0}}$, recall from Lemma~\ref{lemma:slope_Gieseker_relation} that, if $\sigma=(\sigma_1,\ldots,\sigma_{j_0})$ with $\sigma_j \in \mathbb{Q}_{\geq 0}$ not all zero, any $\sigma$-semistable torsion-free coherent sheaf $E$ is slope-semistable with respect to $ \gamma = \sum\nolimits_j \sigma_j c_1(L_j)^{n-1} \in N_1(X)_\mathbb{\mathbb{Q}}$. For the remainder of the section, we suppose that $\gamma$ belongs to $\mathrm{Pos}(X)_\mathbb{R}$. 

Fixing topological invariants, a moduli space $\mathcal{M}_\sigma$ for $\sigma$-semistable sheaves exists by Theorem~\ref{thm:multiGieseker_moduli_exist}, and we denote by $\mathcal{M}_\sigma(\Lambda)$ the subscheme of $\mathcal{M}_\sigma$ that parametrises sheaves with determinant line bundle $\Lambda$. On the other hand, it follows from Proposition~\ref{prop:a_cic_in_every_chamber} and Theorem~\ref{thm:slopemoduli} that there exists a semi-normal moduli space $\mathcal{M}^{\mu\text{-}ss}(\Lambda)$ for sheaves of the chosen topological type with  determinant $\Lambda$ that are slope-semistable with respect to $\gamma$. As in the classical case \cite{LePotierDonaldsonUhlenbeck, JunLiDonaldsonUhlenbeck}, these moduli spaces are related as follows:

\begin{proposition}[The multi-Gieseker-to-Uhlenbeck morphism]\label{prop:Gieseker_to_DU_morphism}
 In the situation above, there exists a natural morphism  
\[\Psi_{DU}: \mathcal{M}_\sigma(\Lambda)^{w\nu} \to \mathcal{M}^{\mu\text{-}ss}(\Lambda),\]
mapping an $S$-equivalence class of multi-Gieseker-semistable sheaves represented by the Gie\-se\-ker-polystable sheaf $E$ to $\Phi_{E}(pt.)$, where $\Phi_E: \{pt.\} \to \mathcal{M}^{\mu\text{-}ss}(\Lambda)$ is the classifying map associated with the trivial family with fibre $E$ over a point. 
If $\mathcal{C}$ is a component of $\mathcal{M}_\sigma(\Lambda)$ whose general point represents a slope-stable locally free sheaf, then $\mathcal{C}$ is mapped birationally onto its image, which is a component of $\mathcal{M}^{\mu\text{-}ss}(\Lambda)$.
\end{proposition}Before we give a proof of Proposition~\ref{prop:Gieseker_to_DU_morphism} in Section~\ref{subsect:GDU_proof} below, we first establish some general results concerning the interplay of good quotients and semi-normalisations needed for the proof.

\subsection{Good quotients and semi-normalisations}
In the following, by \emph{variety} we mean a reduced separated scheme of finite type over the complex numbers, in particular, a variety is potentially reducible.
\begin{lemma}\label{lem:weaklynormalquotients}
 Let $G$ be a complex reductive Lie group, and let $X$ be a semi-normal algebraic $G$-variety that admits a good quotient $\pi: X \to X\hq G$. Then, $X\hq G$ is semi-normal.
\end{lemma}
\begin{proof}
 Let $\mu: X^\nu \to X$ be the normalisation of $X$. The group $G$ acts on $X^\nu$ in such a way that $\mu$ is $G$-equivariant. Then, the good quotient $\pi^\nu: X^\nu \to X^\nu\hq G$ exists, and $X^\nu\hq G$ is  a normalisation of $X\hq G$.  Moreover, by the universal properties of $\pi^\nu$ and of the semi-normalisation $(X\hq G)^{w\nu} \to X \hq G$, we obtain the following commutative diagram:
\[\begin{xymatrix}{
   X^\nu  \ar[r]^{\pi^\nu} \ar[d]_\mu &   X^\nu\hq G \ar[d]_{\mu_G}\ar[r] &   (X\hq G)^{w\nu} \ar[ld] \\
 X   \ar[r]^\pi    &      X \hq G  . &  
}
  \end{xymatrix}
\]
The map $\pi^\nu$ maps $\mu$-fibres to $\mu_G$-fibres; in other words, 
\begin{equation}\label{eq:fibresgotofibres}
 \pi^{\nu}(\mu^{-1}(x)) \subset \mu_G^{-1}(\pi(x)) \quad \quad \text{for all } x \in X.
\end{equation}

Recall from \cite[Sect.~2.3]{GrebToma} that $X\hq G$ is semi-normal if and only if the complex space associated with $X\hq G$ is semi-normal. Let $U \subset X\hq G$ be an open subset (in the Euclidean topology), and let $f$ be holomorphic on $\mu_G^{-1}(U)$ and constant on $\mu_G$-fibres. Owing to the construction of the semi-normalisation, see e.g.~\cite[\S 15]{RemmertSCV_VII}, to prove our claim it suffices to show that $f$ descends to a holomorphic function on $U$.

The pull-back $(\pi^\nu)^*(f) =: \hat f$ of $f$ via $\pi^\nu$ is $G$-invariant and holomorphic on $(\pi^\nu)^{-1}(\mu_G^{-1}(U))$, and in addition constant on $\mu$-fibres by \eqref{eq:fibresgotofibres}. As $X$ is semi-normal by assumption, and as $\mu$ is $G$-equivariant, $\hat f$ descends to a $G$-invariant holomorphic function on $\pi^{-1}(U)$. Since $\pi$ is a categorical quotient also in the analytic category by \cite{LunaFonctionesInvariantes}, there exists a holomorphic function $\tilde f$ on $U$ such that $\pi^*(\tilde f) = \hat f$. By construction, $(\mu_G)^*(\tilde f) = f$, i.e., $f$ descends to $U$, as desired. 
\end{proof}

\begin{lemma}[Taking quotients commutes with seminormalisation]\label{lem:quotients_and_seminormalisations_commute}
  Let $G$ be a complex reductive Lie group, and let $X$ be an algebraic $G$-variety which admits a good quotient $\pi: X \to X\hq G$. Then, the good quotient $X^{w\nu}\hq G$ of the action of $G$ on the semi-normalisation $X^{w\nu}$ exists and $X^{w\nu}\hq G$ is isomorphic to the semi-normalisation $(X \hq G)^{w\nu}$.
\end{lemma}
\begin{proof}
Let $\mu: X^{w\nu} \to X$ be the semi-normalisation. The composition $\pi \circ \mu$ is a $G$-invariant affine morphism, and hence it follows from \cite[Thm.~1.1]{AlperEaston} that the good quotient $X^{w\nu}\hq G$ exists. It follows from \cite{LunaFonctionesInvariantes} that every good quotient is also a semistable quotient in the sense of \cite{HMP}. Additionally, we note that $X^{w\nu}\hq G$ is semi-normal by Lemma~\ref{lem:weaklynormalquotients} above. Based on the universal properties of semi-normalisations and categorical quotients, we obtain the following commutative diagram
\[\begin{xymatrix}
   { X \ar[d]_\pi& \ar[l] X^{w\nu} \ar[d] \\
X \hq G & X^{w\nu} \hq G \ar[l]_\psi \ar[ld]^\varphi\\
(X\hq G)^{w\nu}. \ar[u]& 
}
  \end{xymatrix}
\]

As semi-normalisation maps are homeomorphisms, and since semistable quotiens are also categorical quotients in the topological category \cite[p.~235]{HMP}, we conclude that both $\psi$ and $\varphi$ are homeomorphic. But then $\varphi$ is biregular by \cite[(15.4)]{RemmertSCV_VII}, as claimed.
\end{proof}

\subsection{Proof of Proposition~\ref{prop:Gieseker_to_DU_morphism}}\label{subsect:GDU_proof}
For this, we have to delve a little deeper into the construction of $\mathcal{M}_\sigma$: The moduli space $\mathcal{Q}$ of $\sigma$-semistable representations of the quiver $Q$ with appropriate dimension vector is the GIT-quotient of an open subset $V^{\sigma\text{-}ss} \subset V$ of a finite dimensional $G$-module, where $G$ is a product of general linear groups, cf.~\cite{King}. By Theorem~\ref{thm:semistabilitycomparison} above and \cite[Prop.~4.2]{ConsulKing}, there exists a locally closed $G$-invariant subscheme $R^{\sigma\text{-}ss}$ of $V^{\sigma\text{-}ss}$ with induced GIT-quotient $R^{\sigma\text{-}ss} \hq G = \mathcal{M}_\sigma \subset \mathcal{Q}$. It hence follows that $\mathcal{M}_\sigma(\Lambda)$ appears as a GIT-quotient $R\hq G$ of a $G$-invariant subscheme $R$ of $R^{\sigma\text{-}ss}$. 

Moreover, it is known \cite[Prop.~4.2]{ConsulKing} that $R^{\sigma\text{-}ss}$ and hence $R$ carries a universal $G$-equivariant flat family $\mathscr{F}$ of $\sigma$-semistable sheaves. Pulling this family back to the semi-normalisation $R^{w\nu}$ and using Lemma~\ref{lemma:slope_Gieseker_relation}, we obtain a $G$-equivariant flat family $\mathscr{F}^{w\nu}$ of $\gamma$-slope-semistable sheaves parametrised by the semi-normal scheme $R^{w\nu}$. The universal property of $\mathcal{M}^{\mu\text{-}ss}(\Lambda)$ hence yields a $G$-invariant morphism $R^{w\nu} \to \mathcal{M}^{\mu\text{-}ss}(\Lambda)$. As $R^{w\nu}\hq G \cong (R\hq G)^{w\nu} = \mathcal{M}_\sigma(\Lambda)^{w\nu}$ by Lemma~\ref{lem:quotients_and_seminormalisations_commute} above, this map descends to give the desired morphism $\Psi_{DU}: \mathcal{M}_\sigma(\Lambda)^{w\nu} \to \mathcal{M}^{\mu\text{-}ss}(\Lambda)$.  

Finally, the statement concerning birationality of $\Psi_{DU}$ on special components follows from the separation properties of $\mathcal{M}^{\mu\text{-}ss}(\Lambda)$ discussed in Section~\ref{subsect:geometry_of_slopemoduli} above.  \qed

\addtocontents{toc}{\protect\setcounter{tocdepth}{0}}

\enlargethispage{1cm}
\end{document}